\documentclass[american]{amsart}
\usepackage[T1]{fontenc}
\usepackage[latin1]{inputenc}
\usepackage{amssymb}
\usepackage{babel}
\usepackage[mathscr]{eucal}
\usepackage{todonotes}

 \theoremstyle{plain}
 \newtheorem{cor}{Corollary}

\theoremstyle{plain}
 
 \theoremstyle{plain}
 \newtheorem{prop}{Proposition}
 \theoremstyle{plain}
 
\theoremstyle{plain}
 \newtheorem*{thm*}{Theorem}

\theoremstyle{definition}

\newtheorem*{question}{Question}

\theoremstyle{remark}

\newtheorem*{acknowledgement*}{Acknowledgement}

\newcommand{\RR}{\mathbb{R}}\newcommand{\QQ}{\mathbb{Q}}\newcommand{\EE}{\mathbb{E}}
\newcommand{\PP}{\mathbb{P}}
\newcommand{\NN}{\mathbb{N}}
\newcommand{\ZZ}{\mathbb{Z}}

\newcommand{\Qp}{\QQ_{p}}

\newcommand{\affR}{\mathrm{Aff}(\RR)}\newcommand{\affO}[1]{\mathrm{Aff}(#1 )}
\newcommand{\affQ}{\mathrm{Aff}(\QQ)}

\newcommand{\inv}{^{-1}}

\newcommand{\bnd}{\mathbf{bnd}}

\newcommand{\act}{\star}
\newcommand{\bigslant}[2]{{\raisebox{.2em}{$#1$}\left/\raisebox{-.2em}{$#2$}\right.}}

\begin{document}

\begin{flushright}\today\end{flushright}

\title[A construction of the measurable Poisson boundary]{A construction of the measurable Poisson boundary: from discrete to continuous groups}
\author{Sara Brofferio}
\address{S. Brofferio\\ Universit\'e Paris-Sud\\
Laboratoire de Math\'ematiques et IUT de Sceaux \\
91405 Orsay Cedex\\ France }
\email{sara.brofferio@u-psud.fr}

\begin{abstract}
Let $\Gamma$ be a dense countable subgroup of a locally compact continuous group $G$. Take a probability measure $\mu$ on $\Gamma$. There are two natural spaces of harmonic functions: the space of $\mu$-harmonic functions on the countable group $\Gamma$ and the space of $\mu$-harmonic functions seen as functions on $G$ defined a.s. with respect to its Haar measure $\lambda$.  This leads to two natural Poisson boundaries: the $\Gamma$-Poisson boundary and the $G$-Poisson boundary. Since boundaries on the countable group are quite well understood, a natural question is to ask how $G$-boundary is related to the $\Gamma$-boundary.

In this paper we present a theoretical setting to build the $G$-Poisson boundary from the $\Gamma$-boundary. We apply this technics to build the Poisson boundary of the closure of the Baumslag-Solitar group in the group of real matrices. In particular we show that, under moment condition and in the case that the action on $\RR$ is not contracting, this boundary  is the $p$-solenoid.
\end{abstract}

\maketitle

An important field in the studies of random walks on groups investigates harmonic functions, that is, given a measure $\mu$ on a group $\Gamma$, describe the functions $f$ on the group such that
\begin{equation}\label{eq:fonct-harm-int}
f(g)=\int_\Gamma f(g\gamma)d\mu(\gamma).
\end{equation}
The Poisson Boundary is, in this setting, the measurable space that gives the integral representation of all bounded harmonic functions. This spaces can also be interpreted as the asymptotic information contained in paths of the random walk of law $\mu$. A natural question is to determine when this space is trivial and, if it is not,  to identify  a  geometrical model.

After the works of Blakwell, Choquet and Deny on abelian groups and the seminal papers of Furstenberg in the sixties, many progress have been made. In particular when the harmonic functions live on a \emph{countable discrete group} $\Gamma$, a complete theory has been developed from the works of Derriennic \cite{Der}, Kaimanovich and Vershik \cite{KV} that allows to construct the Poisson Boundary (or at least decide whether it is  trivial) for large classes of groups.

In the more general cases where the measure $\mu$ is supported by a locally compact group $G$, the situation is more complex and one has to decide on which space  harmonic functions live. A natural choice is to consider  harmonic functions as a sub-space of  $L^\infty (G, \lambda)$, the space of essentially bounded functions with respect to the Haar measure $\lambda$ of the group. If the measure $\mu$ is spread-out (thus well adapted to the continuous structure)  satisfying general results have been obtained for Lie group.
 The more general case, when the measure $\mu$ is not necessarily smooth, is far to be completely understood. Some results have been obtained for particular classes of groups (e.g. Nilpotent groups \cite{Gui73, Bre}, NA groups \cite{Rau}..). Abstract constructions have been also proposed, but they do not  allow in general to construct geometrical model for the boundary nor to verify if it is trivial. I refer to the survey of M.Babillot \cite{Bab} for a precise and complete presentation on the subject and a more detailed bibliography.

The opposite case to $\mu$ being smooth  arises when the measure $\mu$ is purely atomic and, thus, supported by a countable subgroup $\Gamma$, that we can suppose dense in the continuous group $G$. In this situation, harmonic functions can be seen both as functions on the discrete group $\Gamma$ and as measurable functions on the continuous group $(G,\lambda)$.

 When the Poisson boundary of the discrete group $(\Gamma, \mu)$ is known (that is we can describe $\Gamma$-harmonic functions), several natural questions concerning $G$-measurable harmonic functions  arise:
 \begin{itemize}
\item Which $\Gamma$-harmonic function can be extended to $G$-harmonic function?
\item How are related the $\Gamma$-Poisson Boundary and the $G$-Poisson Boundary?
\item If we know how $G$ acts on  the $\Gamma$-measurable Poisson Boundary, is it possible to determine conditions that imply that there are no $G$-harmonic function?
\end{itemize}

The goal of this manuscript is to investigate these questions. We are in particularly interested in the case of the groups of matrices with rational entries seen as subgroups of real matrices. In this case the Poisson boundaries of the countable groups are well understood (\cite{BS}), while there still a lot of open questions concerning the Poisson boundaries of the corresponding real groups (see section \ref{sect: exmeples} for more detailed examples).

In section \ref{sect: construction G-bound},  we give a general construction of the $G$-Poisson boundary as $\Gamma$-ergodic components in the product of $G$ and  the $\Gamma$-boundary (Proposition \ref{prop:GammatoG-PB}). We use this construction to exhibit the real boundary in the case of the Baumslag Solitar group $BS(1,p)$ seen as a dense subgroup  of $$\left\{\left[
                                     \begin{array}{cc}
                                       p^{m} & b \\
                                       0 & 1 \\
                                     \end{array}
\right]| m\in\ZZ, b\in\RR\right\}= \RR\rtimes\ZZ.$$
In particular, if $\mu$ is dilating on $\RR$ it is known that the  $BS(1,p)$- Poisson boundary is the $p$-adic field $\QQ_p$ (thus there is  no "real" competent in the boundary), however the real Poisson boundary is not trivial and is given by the $p$-solenoid
$$[0,1)\times \ZZ_p=\bigslant{(\RR\times \QQ_p)}{\ZZ}$$
where the action of $\ZZ$ on $\RR\times \QQ_p$ is the diagonal action (Corollary \ref{cor-PBofBS}).

\begin{acknowledgement*}
  I would like to thanks Vadim Kaimanovich, Jean-Fran\c cois Quint and Bertrand Deroin for the illuminating discussions.

  A spacial thanks to Wolfgang Woess for his continuous and kind support during all my mathematical carrier.
\end{acknowledgement*}

\section{$G$-harmonic functions and $G$-Poisson boundary}
We present in this section a brief introduction to measurable Poisson boundary following the notation of Babillot \cite{Bab} and Kaimanovich \cite{Kai00}
\subsection{$G$-harmonic functions}
Let $G$ be locally compact second countable (thus metrizable and  complete) group.
Let $\mathfrak{G}$ be the Borel $\sigma$-algebra of $G$ and $\lambda$ the right Haar measure.

Let $\mu$ be a probability on $G$ such that the closed semigroup generated by its support is the whole group $G$.

We say that a function $f\in L^\infty(G,\lambda)$ is \emph{$\mu$-harmonic on} $(G,\lambda)$ (or \emph{$G$-harmonic}) if
$$f(g)=\int_\Gamma f(g\gamma)d\mu(\gamma)\qquad \mbox{ for $\lambda$-almost all } g\in G.$$
We denote by $H^\infty_{\lambda}(G)$ the subspace $G$-harmonic functions in $L^\infty(G,\lambda)$.

It can be shown, convolving on the left by an identity approximation on $G$, that any $f\in H^\infty_\lambda(G)$ is $\lambda$-almost sure limit of harmonic functions that are  left uniformly continuous on $G$.  In this sense the space of $G$-harmonic functions is determined by the behavior of continuous ones. In particular if all continuous harmonic functions are constant then $H^\infty_\lambda(G)$ is trivial. We denote by $H_{\mathrm{luc}}^\infty(\Gamma)$ the space of left uniformly continuous  $G$-harmonic functions.

\subsection{Random walks and invariant map}
Harmonic functions can be seen as asymptotic values of random walk in the following way.
Let $(\Omega,\PP)=(G,\mu)^\NN, $ be the space of random steps  and consider the right random walk
$$r_n(\omega)=\omega_1\cdots\omega_n.$$
Let $f$ be a bounded $G$-harmonic function. Remarks that since the function $f$ is defined only $\lambda$-almost surely, the process $f(gr_n(\omega))$ is well defined only for $\lambda$-almost all $g$. For this reason one need to chose the starting point $g$ according to $\rho$, a probability law on $G$ absolutely continuous with respect to $\lambda$. On the space $(G\times\Omega,\rho\times\PP)$ then  the random process $f(gr_n(\omega))$ is well defined and, since  $f$ is harmonic, it is a bounded martingale. Thus the limit
\begin{equation}\label{eq:f->Z}
\lim_{n\to \infty} f(gr_n(\omega))=: Z_f(g,\omega)\mbox{ exists $\rho(dg)\PP(d\omega)$-almost surely. }
\end{equation}

Let $T$ be the shift on $\Omega$ then is easily checked that
$$Z_f(g,\omega)=Z_f(g\omega_1,T\omega) \qquad \mbox{$\rho(dg)\PP(d\omega)$-almost surely}$$
that is $Z_f$ is  a bounded measurable \emph{invariant map} of $G\times \Omega$. In fact (\ref{eq:f->Z}) defines an isometry of $H^\infty_\lambda(G)$ onto the subspace of measurable invariant map in $L^\infty(G\times\omega,\rho\times\PP)$, whose the reverse map is given by
$$f_Z(g):=\EE(Z(g,\omega)) \qquad\mbox{ $\rho(dg)$-almost surely}.$$

\subsection{Poisson transform and $G$-Poisson boundary}
 Take a  measurable space $(X,\mathfrak{X}, \nu)$  endowed with a measurable  $G$-action and  a $\mu$-invariant probability measure $\nu$. The \emph{Poisson transform}
$$\mathcal{P}_\nu: \phi \mapsto f_\phi(g):=\int\phi(g \cdot  x)d\nu(x)$$
maps  any bounded function $\phi$ in $L^\infty(X,\rho*\nu)$ to a $\mu$-harmonic function $f_\phi$ of  $H_\lambda^\infty(\Gamma)$. \\
Observe that the Poisson transform would not be well defined as map of $L^\infty(X,\nu)$. In fact, since $\nu$ is not in general $G$-quasi invariant (i.e. $g*\nu$ is not in general absolutely continuous with respect to $\nu$), two functions that coincide $\nu$-a.s. can have different images.

If the Poisson transform is an \textbf{isometry} of $L^\infty(X,\rho*\nu)$ \textbf{onto} $H_\lambda^\infty(G)$ then $(X,\nu)$ is called  \emph{$(G,\mu)$-Poisson boundary}. It can be shown that the Poisson boundary is unique as a $G$-measurable space.

If $X$ is the $G$-Poisson boundary then there exists a measurable \emph{boundary map}
$\mathbf{bnd}:\Omega\to X$ such that for every harmonic function $f\in H^\infty_{\lambda}(G)$ there exits $\phi_f\in L^\infty(X,\rho*\nu)$ such that
$$\phi_f(g\cdot \mathbf{bnd}(\omega))=\lim_{n\to \infty} f(gr_n(\omega)) \qquad \rho(dg)\PP(d\omega)-\mbox{a.s.}.$$
Thus
\begin{eqnarray*}
  f(g) &=& \int\phi_f(g \cdot  x)d\nu(x) \qquad \rho(dg)-\mbox{a.s.} \qquad \mbox{and}\\
  Z_f(g,\omega) &=& \phi_f(g\cdot \mathbf{bnd}(\omega)) \qquad \rho(dg)\PP(d\omega)-\mbox{a.s.}.
\end{eqnarray*}

The $\mu$-invariant measure $\nu$ on $X$ is then the image of $\PP$ under $\mathbf{bnd}$. The boundary map is $G$-equivariant in the sense that
$\mathbf{bnd}(\omega)=\omega_1\cdot\mathbf{bnd}(T\omega)$.

\subsection{Countable group $\Gamma$ }
   A special case arise when the group  $G=\Gamma$ is countable. The Haar measure $\lambda$ is then the counting measure and one can chose $\rho$ to have a mass in all elements $g\in \Gamma$. This mean that all the equalities above hold for all $g\in \Gamma$.

   In this particular case (and under the hypothesis that the support of $\mu$ generate $\Gamma$ as a semigroup) the stationary measure  $\nu$ on $X$ is $\Gamma$-quasi invariant  and $\mathcal{P}_\nu$ is well defined on $L^\infty(X,\nu)$  itself.

    The fact that that $\mu$ is absolutely continuous with respect to $\lambda_\Gamma$ is also  fundamental on the study of Poisson boundary based on entropy \cite{Der,KV}. This complete theory has permitted to determine a geometrical model of the Poisson boundary.

\subsection{Countable subgroup $\Gamma$ of a continuous $G$}
In this note we are interested on the case when the measure $\mu$ is supported on countable subgroup $\Gamma$ of a continuous group $G$ and in particular when  $\Gamma$ is dense in $G$. Then  a \emph{continuous} harmonic function $f$ on $G$ is uniquely determined by the values  $f(\gamma)$ for $\gamma\in\Gamma$. Thus $f$ can also be seen as a $\Gamma$-harmonic function. In other words, the restriction to $\Gamma$ is an isometric embedding of $H_{\mathrm{luc}}^\infty(G)$ into $H_{\lambda}^\infty(\Gamma)$. In particular if $(X,\nu)$ is the $\Gamma$-Poisson boundary then there exists $\phi$ in  $L^\infty(X,\nu)$ such that
 $$ f(\gamma)=\int_X \phi (\gamma\cdot x)\qquad\forall \gamma\in \Gamma.$$
 However this integral representation do not hold in general for $f(g)$ when $g$ is not in $\Gamma$, since $X$ is not a priori a $G$-space.

 In conclusion the $\Gamma$-Poisson boundary contains in principle all the information about the $G$-Poisson boundary. But in order to obtain this information one should be able to answer to two related questions:
 \begin{itemize}
   \item Determine the $G$-action on (an extension of) $X$  adapted to the action of $G$ on $H_{\lambda}^\infty(G)$
   \item Determine which are the functions in $L^\infty(X,\nu)$ whose Poisson transform  can be extended to $G$.
 \end{itemize}

\subsection{Examples: Linear groups with rational coefficients}\label{sect: exmeples}
We are in particular interested in the case where the $\Gamma$-Poisson boundary is known, but the understanding of $G$-harmonic functions are not completely understood. Here some examples.

\subsubsection{Affine groups}
The real affine group $\affR$ is the group of real map $(b,a):x\mapsto ax+b$ with $a\in\RR^*_+$ and $b\in \RR$ that is the group of matrices
$$\affR =\left\{ \left[
                                     \begin{array}{cc}
                                     a & b \\
                                       0 & 1 \\
                                     \end{array}
\right]| a\in\RR^*_+,  b\in \RR \right\}=\RR \rtimes \RR^*_+.$$

 Harmonic functions on  $\affR$ have been widely studied and some results are known also without continuous hypothesis on the measure $\mu$. In particular under $\log$-moment hypothesis
$$\EE(|\log a|)<\infty\mbox{ and }\EE(\log^+ b)<\infty$$
it is known that:
\begin{itemize}
  \item If $\EE(\log a)=0$ the $\affR$-Poisson Boundary is trivial (\cite{Rau2}, see also \cite[sect.4.5]{Bab})
  \item If $\EE(\log a)<0$ the $\affR$-Poisson Boundary is $\RR$ with the $\mu$-invariant measure $\nu$ given by the law of
      \begin{equation}\label{eq-Zinf}
      Z_\infty=\sum_{n=1}^\infty a_1\cdots a_{n-1}b_n
      \end{equation}(\cite{Rau}, see also \cite[thm 5.7]{Bab})
\end{itemize}
If $\EE(\log a)>0$ and the measure is spread-out then $\affR$-Poisson Boundary is trivial. But the question is still open on what happen in general if  $\EE(\log a)>0$, in particular when the measure $\mu$ is supported by a countable subgroup $\Gamma$.

On the other hand, using entropic criteria, the $\Gamma$-Poisson boundaries are well understood.   If $\Gamma=\affQ$, the group of affine map with rational coefficient and under suitable moment conditions, the $\affQ$-Poisson boundary is given by the product of the $p$-adic fields $\QQ_p$ where the sum (\ref{eq-Zinf}) converges a.s., that is
$$\prod_{p: \EE(\log |a|_p)<0} \QQ_p,$$
where we use the convention that $\QQ_\infty=\RR $ (see \cite{Bro}).

This property was first proved by V.Kaimanovich \cite{Kai00} in the case of the Baumslag-Solitar group that is $BS(1,p)$, i.e.
$$BS(1,p)=\left\langle \left[
                                     \begin{array}{cc}
                                       p^{\pm1} & \pm1 \\
                                       0 & 1 \\
                                     \end{array}
\right] \right\rangle=\left\{ \left[
                                     \begin{array}{cc}
                                     p^m & qp^n \\
                                       0 & 1 \\
                                     \end{array}
\right]| m,n\mbox{ et }q\in\ZZ\right\}=\ZZ(\frac1p) \rtimes \ZZ.$$
for some prime $p$. In this particular case the $BS(1,p)$-Poisson boundary is $\RR$ if $\EE(\log a)=-\EE(\log|a|_p)<0$ and $\QQ_p$ if $\EE(\log|a|_p)=-\EE(\log a)<0$.

 A natural question is then what are the harmonic functions that can be extended to (continuous) harmonic functions of the closure of $BS(1,p)$ in $\affR$, that is
to
$$\affO{p,\RR} =\left\{ \left[
                                     \begin{array}{cc}
                                     p^m & b \\
                                       0 & 1 \\
                                     \end{array}
\right]| m\in\ZZ, b\in \RR \right\}=\RR \rtimes \ZZ.$$

It turn out that even if $BS(1,p)$-Poisson boundary is $\QQ_p$, the real Poisson boundary is not trivial. In Corollary \ref{cor-PBofBS} we will construct the $\affO{p,\RR}$-Poisson boundary as a $p$-solenoid.

J.-F. Quint presented in the unpublished manuscript \cite{Qui} a similar example of dynamical system acting in non contacting way on the torus and constructed   harmonic functions on the unstable variety.

As we will see in Corollary \ref{cor-fD} this kind of construction is possible since the action of $BS(1,p)$  on  $\affO{p,\RR}\times \QQ_p$ has a discrete orbit.
It is still not clear to me what may happen when the action of $\Gamma$ on the product of $G$ and the $\Gamma$-Poisson boundary is dense.
\begin{question}
 For instance, let $\affO{1/2,1/3}$ be the countable subgroup generated by the affinities
 $$\left\langle \left[
                                     \begin{array}{cc}
                                       3^{\pm1} & \pm1 \\
                                       0 & 1 \\
                                     \end{array}
\right],
\left[
                                     \begin{array}{cc}
                                       2^{\pm1} & \pm1 \\
                                       0 & 1 \\
                                     \end{array}
\right]
 \right\rangle=\left\{ \left[
                                     \begin{array}{cc}
                                     2^{m_2}3^{m_3} & q2^{n_2}3^{n_3} \\
                                       0 & 1 \\
                                     \end{array}
                                     \right]| m_i,n_i\mbox{ et } q\in\ZZ\right\}. $$
 Suppose  $\EE(\log |a|_\infty)>0$, thus  $\Gamma$-Poisson boundary is  equal to $\QQ_2$, $\QQ_3$ or $\QQ_2\times\QQ_3$ (according to the sign of $\EE(\log |a|_2)$ and $\EE(\log |a|_3)$ )  and has no real component.  Is then the $\affR$-Poisson boundary trivial?
\end{question}

\subsubsection{Semi-simple groups}
Similar questions arise in semi-simple situations. Take, for instance, a measure $\mu$ supported by $SL_2(\QQ)$. Then the $SL_2(\QQ)$-boundary is the product the $\QQ_p$-projective lines for all prime $p$  such that the support of $\mu$ is not contained in a compact subgroup of $SL_2(\QQ_p)$ (see \cite{BS}).  In particular for
$$\Gamma=SL_2(\ZZ(1/2))=\left\{  \left[
                                     \begin{array}{cc}
                                     a & b \\
                                       c & d \\
                                     \end{array}
\right] | ad-cd=1, a,b,c\mbox{ et }d\in \ZZ/2^m \mbox{ for some } m\in\ZZ  \right\}$$
the $\Gamma$-Poisson boundary is $\PP^1(\RR)\times\PP^1(\QQ_2)$. It would seem natural to say that the $SL_2(\RR)$-Poisson boundary should be $\PP^1(\RR)$, however I am not aware of any proof of this fact.

See also \cite{Bab} section 1.7.4, for a similar example.

\section{From $\Gamma$-boundaries to  $G$-boundaries}\label{sect: construction G-bound}

\subsection{Construction of a $G$-action on a $\Gamma$-space}
Let $(X, \mathfrak{X},  \nu)$ be a $\Gamma$-measurable Lebesgue space equipped with a measure $\nu$ that is $\Gamma$-quasi invariant. Suppose that $\Gamma$ is contained in locally compact group $G$. We want to construct a sort of minimal class of functions on $X$, on which $G$ acts in such a way that the restriction to $\Gamma$ of this action coincides with the $\Gamma$-action.

Consider the product  space $(G\times X, \mathfrak{G}\times \mathfrak{X}, \rho\times \nu)$ and define the $\Gamma$-action on $G\times X$
\begin{equation}\label{eq-act-Gamma}
    \gamma   \act (g,x):=(g\gamma\inv, \gamma\cdot x).
\end{equation}

Let $\mathfrak{I}$ be the $\sigma$-algebra of $(\Gamma,\act)$-invariant functions of $G\times X$ that is the class of the functions $\phi$ such that $\rho(dg)\times\nu(dx)$-almost surely
$$\phi(g,x)=\phi(g\gamma\inv, \gamma x) \qquad \forall \gamma\in \Gamma.$$

The $\sigma$-algebra $\mathfrak{I}$ is complete because $\rho\times \nu$ is $(\Gamma,\act)$-quasi invariant and $\Gamma$ is countable. By  Rokhlin's correspondence, we have a partition $\eta$ of $G\times X$ associated to the $\sigma$-algebra $\mathfrak{I}$. Let $\widetilde{X} =G\times X/\eta$ be the quotient space. Then $L^\infty(G\times X,\mathfrak{I},\rho\times \nu)= L^\infty(\widetilde{X}, \widetilde{\rho\times \nu})$. As we wanted
$\widetilde{X}$ has a natural structure of $G$-space inherited by the left multiplication on the $G$ component
      $$g_0\eta(g,x)=\eta(g_0g,x)$$
      that coincide with l'action of $\Gamma$ on $X$ in the sense that
      $$\gamma\eta(e,x)=\eta(\gamma,x)=\eta(e,\gamma\cdot x).$$

In the next section ,using this measure theoretical construction, we will build the $G$-boundary on the $\Gamma$-boundary and prove that, in the case the $\act$-action has a fundamental domain, this fundamental domain is  the $G$-boundary.  Still,  I do not understand how to construct a geometrical model of this measure space in the case the $\Gamma$-action is "dense".

An interesting case is, for instance, when  $G$ acts on $X$ and this action coincides with the $\Gamma$-action. Then $L^\infty(X,\rho*\nu)$ embed isometrically in $L^\infty(G\times X,\mathfrak{I},\rho\times \nu)$. In fact if $\psi\in L^\infty(X,\rho*\nu)$
      then
      $$\phi_\psi(g,x):=\psi(gx)$$
      is clearly $\Gamma$-invariant and this embedding  is an isometry since
      $$\|\phi_\psi\|_\infty=\lim_{p\to\infty}\big(\int \phi_\psi(g,x)^p \rho(dg) \nu(dx)\big)^{1/p}=\lim_{p\to\infty}\big(\int \psi(y)^p \rho*\nu(dy)\big)^{1/p}=\|\psi\|_\infty.$$
\begin{question}
However it is not clear under which conditions this map is surjective, that is when the $\widetilde{X}$ coincide with $X$.

For instance, as a toy model, take $G=(\RR,+)$, $X=\RR$ and $\Gamma=\QQ$. For which measure $\nu$ does $\widetilde{X}=\RR$? This is true by if $\nu$ is a.c. with respect to the Lebesgue measure, but what happen for other measures?

What happen if  $G=SL_2(\RR)$ $X=\PP^1(\RR)$ and $\Gamma=SL_2(\ZZ(1/2))$ (or $\Gamma=SL_2(\QQ)$)?
\end{question}

\subsection{From  $\Gamma$-boundaries to  $G$-boundaries}

Suppose that the measure $\nu$ on $X$ is $\mu$-invariant.
For every bounded function $\phi$ in $L^\infty(G\times X,\rho\times \nu)$  define the Poisson transform:
$$\mathcal{P}_\nu: \phi \mapsto f_\phi(g)=\int\phi(g,x)d\nu(x)\mbox{for $\lambda(dg)$-almost all } g. $$

As we wanted,  if $\phi\in\mathfrak{I}$ then  $f_\phi$ is a bounded $\mu$-harmonic function on $L^\infty(G,\lambda)$. In fact
\begin{eqnarray*}
f_\phi(g)&=&\int\phi(g,x)\nu(dx)=\int\phi(g,\gamma\cdot x)\nu(dx)\mu(d\gamma)=\\
   & =&\int\phi(g\gamma ,x)\nu(dx)\mu(d\gamma)=\int f_\phi(g\gamma)\mu(d\gamma).
\end{eqnarray*}

The following proposition shows that all $G$-harmonic functions can be written in such a way
\begin{prop}\label{prop:GammatoG-PB}
   If $(X,\nu)$ is the Poisson boundary of $(\Gamma,\mu)$ then for every $\mu$-harmonic function on $G$ there exists a bounded function $\phi\in\mathfrak{I}$ such that $f=f_\phi$ in $L^\infty(G,\lambda)$.\\
      In this case $\mathcal{P}_\nu$ is an isometry form $L^\infty(G\times X,\mathfrak{I},\rho\times \nu)$ onto $H^\infty_\lambda(G)$. In other word $\widetilde{X}$ is the $G$-Poisson boundary.
\end{prop}

\begin{proof}
Let $\omega\in(\Omega,\PP)=(\Gamma^\NN,\mu^{\otimes \NN})$ and $ r_k=r_k(\omega)=\omega_1\cdots \omega_k$ be the right random walk on $\Gamma$ of law $\mu$. The processes $f(g r_k(\omega))$ is a bounded martingale on the space  $(G\times \Omega,\rho\times \PP)$ thus it converges almost surely.  If $\bnd: \Omega\to X$ is the boundary map
\begin{equation}\label{eq-lim-phi}
\lim_{n\to\infty} f(g r_k(\omega))=\phi(g,\bnd(\omega))
\end{equation}
$\rho\times\PP$-almost surely. Thus $\phi(g,\bnd(\omega))$ is $G\times\Omega$ measurable and, since $\nu=\bnd\inv\PP$,  the function $\phi(g, x)$ is $G\times X$-measurable.
Furthermore since $\Gamma$ is countable, for $\rho\times\PP$-almost all $(g,\omega)$
$$ \lim_{n\to\infty} f(g\gamma  r_k(\omega))=\phi(g\gamma,\bnd(\omega))\quad\mbox{for all $\gamma\in \Gamma$.}$$

Observe that $X$ being a $\mu$-boundary we have
$$\omega_1\bnd(T\omega)=\bnd(\omega)$$
$T$ being the shift on $\Omega$.
Take $\gamma_1$ in the support of $\mu$ then the event $ \gamma_1=\omega_1$ as positive measure and conditioned to this event
$$ \phi(g\gamma_1\inv,\gamma_1\bnd(T\omega))=\phi(g\gamma_1\inv,\bnd(\omega))=\lim_{n\to\infty} f(g\gamma_1\inv\gamma_1 r_n(T\omega))=\phi(g,\bnd(T\omega))$$
Since $T\omega$ is independent of $\omega_1$ and of same law as $\omega$ and that the support of $\mu$ generates $\Gamma$, we can conclude that $\phi\in \mathfrak{I}$.

Lets us finely check that the  Poisson transform is an isometry. In fact
      $$\|f_\phi\|_\infty^\rho=\lim_{p\to\infty}\left(\int |f_\phi(g)|^pd\rho(g)\right)^{1/p}
      \leq\lim_{p\to\infty}\left(\int \int|\phi(g,x)|^p d\nu(x) d\rho(g)\right)^{1/p}=\|\phi\|_\infty^{\rho\times\nu}$$
      On the other hand by bounded convergence theorem
       \begin{eqnarray*}
       \|\phi\|_p^{\rho\times\nu} &=& \left(\int \int|\phi(g,x)|^p d\nu(x) d\rho(g)\right)^{1/p}=\\
       &=&\left(\int \int|\lim_{n\to\infty} f(gr_n(\omega))|^p d\PP(\omega) d\rho(g)\right)^{1/p}=\\
       &=&\lim_{n\to\infty}\left(\int \int| f(gr_n(\omega))|^p  d\rho(g) d\PP(\omega)\right)^{1/p}\\
       &\leq& \left(\int ( \|f\|_\infty^\rho)^p d\PP(\omega) \right)^{1/p} = \|f\|_\infty^\rho
       \end{eqnarray*}
       since $\rho$ is $G$-quasi invariant.
\end{proof}

\subsection{$G$-Poisson boundary as $\Gamma$-ergodic diagonal components}

Another way to express the result of Proposition \ref{prop:GammatoG-PB} is to say that the $G$-Poisson boundary coincides with the space of ergodic components of  $\Gamma$ on $(G\times X)$  with respect to the action $  \act$ defined in (\ref{eq-act-Gamma}).

Observe that the action $  \act$  is, in reality, the standard left diagonal action of $\Gamma$ on $G\times X$:
  $$\gamma\stackrel{\mbox{d}}{\cdot}(g,x)=(\gamma g,\gamma\cdot x).$$
  In fact the two actions are conjugated by the map $\pi: (g,x)\mapsto (g\inv,x)$, that  is an isomorphism of the measure space of $(G\times X,\rho\times \nu)$ that  preserves the class of measure.  Thus the space $L^\infty(G\times X,\mathfrak{I}, \rho \times \nu)$ coincides (via $\pi$) with the space of the bounded functions of $(G\times X, \rho\times \nu)$ that projet on  $\Gamma\setminus(G\times X)$. In particular  the $G$-Poisson boundary is trivial if and only if the (diagonal) action of $\Gamma$ on $(G\times X,\rho\times \nu)$ is ergodic.

Conversely if the action of $\Gamma$ on $G\times X$ is "measurably discrete", that is there exists a fundamental domain $\Delta$, then is possible to  identify the $G$-Poisson boundary with this geometrical model:
\begin{cor}\label{cor-fD}
 Suppose there exists measurable fundamental domain $\Delta\in \mathfrak{G}\times\mathfrak{X}$ for the action $ \act$ of $\Gamma$ on $G\times X$ (or equivalently for the diagonal action) that is
  \begin{itemize}
    \item $\rho\times\nu((G\times X)\setminus \Gamma  \act \Delta)=0$
    \item $\rho\times\nu (\Delta\cap \bigcup_{\gamma\in\Gamma-\{e\}}\gamma   \act \Delta)=0$
   \end{itemize}
Let $\mathfrak{D}$ be the restriction of the $\sigma$-algebra $\mathfrak{G}\times\mathfrak{X}$ to $\Delta$.
Then $L^\infty(\Delta, \mathfrak{D},\rho\times\nu)$ is isometric to $L^\infty(G\times X,\mathfrak{I},\rho\times \nu)$ . The measurable space  $(\Delta, \mathfrak{D})$ with the induced $G$-action
$$g_0*\phi (g,x):=\sum_{\gamma\in \Gamma} \phi(g_0g\gamma\inv,\gamma\cdot x)1_{\Delta}(g_0g\gamma\inv,\gamma\cdot x) \mbox{ for all }\phi\in L^\infty(\Delta, \mathfrak{D},\rho\times\nu) $$
and the $\mu$-invariant measure  defined by
$$\widetilde{\nu}(\phi):=\sum_{\gamma\in \Gamma}\int \phi(\gamma\inv,\gamma\cdot x)1_{\Delta}(\gamma\inv,\gamma\cdot x) \nu(dx)$$
is the $G$-Poisson boundary.
\end{cor}
\begin{proof}
The map $$A \mapsto \Gamma   \act A$$ induces an isometry of $L^\infty(\Delta, \mathfrak{D},\rho\times\nu)$ onto $L^\infty(G\times X,\mathfrak{I},\rho\times \nu).$\\
In fact \emph{if $A$ is a non trivial set of $\Delta$ then $\Gamma   \act A$ is a non trivial set of $\mathfrak{I}$.}
Clearly $\Gamma  \act A\in \mathfrak{I}$ and it has non null measure. Let $B\subset \Delta$ a nontrivial set  such that $\rho\times \nu(A\cap B)=0$. We claim that $\rho\times\nu(\Gamma  \act A\cap \Gamma   \act B)=0$; in fact the measure $\rho\times\nu$ being quasi-invariant of $\rho\times\nu(\gamma  \act A\cap\gamma  \act B)=0$ and if $\gamma_1\neq\gamma_2$
$$\rho\times\nu(\gamma_1 A\cap\gamma_2B)\leq \rho\times\nu (\Delta\cap \bigcup_{\gamma\in\Gamma-\{e\}}\gamma \Delta)=0$$
\emph{The isometry is surjective}. In fact let $I\in \mathfrak{I}$, we claim that $I=\Gamma(I\cap \Delta)$
In fact
$$\Gamma(I\cap \Delta)=\bigcup_\gamma \gamma I\cap\gamma \Delta=\bigcup_\gamma ( I\cap\gamma \Delta)=I\cap \Gamma\Delta$$

Observe that if $A\subseteq \Delta$ then
$$1_{\Gamma\act A}(g,x)=\sum_{\gamma\in \Gamma} 1_{A}(g\gamma\inv,\gamma\cdot x)$$
and the sum has only one term for $\rho\times\nu$-almost all $(g,x)$ .
It easily seen that the projection of $\nu$ on $\mathfrak{D}$ is
 $$ \widetilde{\nu}(A)= \sum_{\gamma\in \Gamma}\int 1_{A}(\gamma\inv,\gamma\cdot x) \nu(dx)=\nu(\Gamma  \act A).$$
\end{proof}

\section{$G$-Poisson boundary of Baumslag-Solitar group}
\begin{cor}\label{cor-PBofBS}
Let $p$ be a  prime number and consider the Baumslag-Solitar group
$$BS(1,p)=\left\langle \left[
                                     \begin{array}{cc}
                                       p^{\pm1} & \pm1 \\
                                       0 & 1 \\
                                     \end{array}
\right] \right\rangle.$$  Let $\mu$ be a irreducible measure on $BS(1,p)$ with first logarithmic moment on $\RR$ and $\Qp$. Suppose that
$$\phi_p=\int_\Gamma \log|a(\gamma)|_pd\mu(\gamma) <0 $$
where $\gamma= \left[
                                     \begin{array}{cc}
                                       a(\gamma) & b(\gamma) \\
                                       0 & 1 \\
                                     \end{array}
\right]$
that is the $BS(1,p)$-Poisson boundary is $X=\QQ_p$.  Let
$$\affO{p,\RR}=\left\{\left[
                                     \begin{array}{cc}
                                       p^{m} & b \\
                                       0 & 1 \\
                                     \end{array}
\right]| m\in\ZZ, b\in\RR\right\}= \RR\rtimes\ZZ$$ be the closure of $BS(1,p)$ in $\affR$.
 Then the $\affO{p,\RR}$-Poisson boundary is the $p$-solenoid :
$$\Delta=\{(g,x)\in \affR\times \QQ_p| a(g)=1; 0\leq b(g) <1; |x|_p\leq 1\}=[0,1)\times \ZZ_p,$$
equipped with the
$\affO{p,\RR}$-action on $\phi\in L^{infty}(\Delta,\rho\times\nu)$:
$$(b,p^m)
\cdot\phi(x_\infty,x_p)= \sum_{\beta \in \ZZ(1/p)} 1_\Delta\cdot\phi(p^{m}x_\infty+b-\beta,p^m x_p+\beta),$$
         and the invariant measure
$$\widetilde{\nu}(\phi):=\sum_{\beta \in \ZZ(1/p)\cap[0,1) }\int \phi(\beta , x-\beta)1_{\ZZ_p+\beta}(x) \nu(dx).$$
\end{cor}

\begin{proof}
We just need to prove that $\Delta$ is a fundamental domain. In fact for any $x\in\QQ_p$ let $\alpha(x)\in \ZZ(1/p)$ such that $|x-\alpha(x)|_p\leq 1$. The choice of $\alpha$ is unique up to the sum with an integer. It easily checked that for every $(b,x)\in \RR\times Q_p$ the unique $k\in \ZZ(1/p)$ such $|x+k|_p\leq 1$ and $b-k\in [0,1)$ is $k=[b+\alpha(x)]-\alpha(x)$. Thus
$$\gamma\act ((b,p^m),x)\in \Delta \Leftrightarrow \gamma=([b+\alpha(p^mx)]-\alpha(p^mx),p^m)$$
\end{proof}

To illustrate how the pervious corollary can be used to  study the behaviors of harmonic functions on $BS(1,p)$, consider for example
$$\phi(g,x)=1_{[0,1)\times\{1\}}(g)1_{p\ZZ_p}(x)$$
and the associated harmonic function:
$$f(b,p^m)=\int \sum_{\beta \in \ZZ(1/p)} 1_{[0,1)}(b-\beta) 1_{p\ZZ_p}(p^m x+\beta)\nu(dx). $$
Then we have
\begin{itemize}
  \item $f$ is periodic of periode $p$ on the $b$ coordinate
  $$f(pk+ b,p^m)=\int \sum_{\beta \in \ZZ(1/p)} 1_{[0,1)}(b-\beta) 1_{p\ZZ_p}(p^m x+\beta+pk)\nu(dx)=f(b,p^m) $$
  \item $\lim_{m\to+\infty} f(b,p^m)=1$ if $b\in[0,1)+p\ZZ$. In fact $\|f\|_\infty=1$ and $b\in [0,1)$
  $$f(b,p^m)\geq \int  1_{[0,1)}(b) 1_{p\ZZ_p}(p^m x)\nu(dx)=\nu(p^{1-m}\ZZ_p)\to 1$$
  when $m\to+\infty$
  \item $\lim_{m\to+\infty} f(b,p^m)=0$ if $b\not\in[0,1)+p\ZZ$ in fact
  \begin{eqnarray*}
    f(b,p^m) &\leq&  \int \sum_{\beta \in \ZZ(1/p)} 1_{[0,1)}(b-\beta) 1_{p\ZZ_p}(p^m x+\beta)1_{p^{1-m}\ZZ_p}(x)\nu(dx) +\\& &\qquad + (1-\nu(p^{1-m}\ZZ_p)) \\
     &\leq &  \sum_{\beta \in \ZZ(1/p)} 1_{[0,1)}(b-\beta) 1_{p\ZZ_p}(\beta) + (1-\nu(p^{1-m}\ZZ_p))\\
     &=&  \sum_{k \in \ZZ} 1_{[0,1)}(b-pk) + (1-\nu(p^{1-m}\ZZ_p))\\
     &=&   1_{[0,1)+p\ZZ}(b) + (1-\nu(p^{1-m}\ZZ_p)).
  \end{eqnarray*}
  \end{itemize}

\end{document}